\documentclass{amsart}
\usepackage{tabularx}
\usepackage[notref,notcite]{showkeys}         % shows labels
\usepackage{enumitem,amsmath,amsfonts,amssymb,xcolor}
\usepackage{latexsym,amsfonts,euscript,amssymb}
\newtheorem{lemma}{\bf Lemma}[section]

\newtheorem{thm}[lemma]{\bf Theorem}
\newtheorem{rem}[lemma]{\bf Remark}

\newtheorem{conj}[lemma]{\bf Conjecure}

\newcommand{\SL}{{\operatorname{SL}}}

\newcommand{\PSL}{{\operatorname{PSL}}}

\newcommand{\PSU}{{\operatorname{PSU}}}

\newcommand{\PSp}{{\operatorname{PSp}}}
\newcommand{\OO}{{\operatorname{P\Omega}}}

 \DeclareMathOperator{\Aut}{Aut}

\input xy
\xyoption{all}

\title[]{Characterization of $\PSL(2,q)$ by the number of singular elements}

\author{Rulin Shen}
\address{Rulin Shen. Department of Mathematics, Hubei Minzu University \\ Enshi, Hubei Province,
    445000, P. R. China. }
\email{shenrulin@hotmail.com}

\author{Deyu Yan}
\address{Deyu Yan. School of Public Health, Hubei University of Medicine \\ Shiyan, Hubei Province,
    442000, P. R. China. }
\email{yandeyu2023@hotmail.com}

\thanks{The research is supported by the NSF of China (Grant No.
 12161035).}

\date{}
\subjclass[2010]{20D06, 20D60} \keywords{Finite groups, $r$-singular
elements, simple groups}
\begin{document}

    \setlength{\parskip}{2mm}

    \maketitle

    \begin{abstract}
        Given a finite group $G$, let $\pi(G)$ denote the set of all primes that divide the order of $G$. For a prime $r \in \pi(G)$,
we define $r$-singular elements as those elements of $G$ whose order is divisible by $r$. Denote by $S_r(G)$ the number of $r$-singluar elements of $G$. We denote the proportion $S_r(G)/|G|$ of $r$-singular elements  in $G$ by ${\mu_r}(G)$. Let $\mu(G) := {\{\mu_r}(G) | r\in \pi(G)\}$ be the set of all proportions of $r$-singular elements for each prime $r$ in $\pi(G)$.
        In this paper,  we prove that if a finite group $G$ has the same set  $\mu(G)$ as the simple group $\PSL(2,q)$,
        then $G$ is isomorphic to $\PSL(2,q)$.
    \end{abstract}

    \section{Introduction}

    All groups
considered in this paper are finite. Given a finite group $G$, let
$\pi(G)$ be the set of all primes that divide the order of $G$. For
each prime $r\in \pi(G)$, we define $r$-singular
    elements as those elements of $G$ whose order is divisible by $r$, while $r$-regular elements are those whose order is not divisible by $r$.
    The number of $r$-singular and $r$-regular elements plays an important role in understanding the structure of the group, particularly in finite simple groups.

    Several researchers have investigated the properties of singular elements in different types of groups. In 1991, Guralnick and L$\rm{\ddot{u}}$beck
    (see \cite{GL}) studied $p$-singular elements in Chevalley groups in
characteristic $p$. In 1995,
Isaacs, Kantor and Spaltenstein (see \cite{IKS}) researched the probability
of an element in a group being $p$-singular. In 2012, Babai, Guest,
    Praeger and Wilson (see \cite{BGP}) investigated the proportions of $r$-regular elements in finite classical groups. Brandl and Shi (see \cite{BS}) characterized
$\PSL(2,q)$ by its set of orders of elements.
    Denote by $S_r(G)$ the number of $r$-singluar elements of $G$. Let ${\mu_r}(G):=$ $S_r(G)/|G|$ denote the proportion of $r$-singular elements
 of $G$.
     Let $\mu(G) := {\{\mu_r}(G) | r\in \pi(G)\}$ be the set of all proportions of $r$-singular elements for each prime $r$ that divides $|G|$. In this paper, we use this set to characterize
    $\PSL(2,q)$. Our main result is the following theorem:

    \begin{thm}
        Let $G$ be a finite group, and $q\ge 4$. If $\mu (G) = \mu
        (\PSL(2,q))$, then $G \cong \PSL(2,q)$.
    \end{thm}

Remark 2.5 in Section 2 says that $\mu(G)=\mu(H)$ implies
 $|G|=|H|$. So $\mu(G)$ contains a lot of information on the structure of $G$. We give the following conjecture.

  \begin{conj}
        Let $G$ be a finite group, and $S$ be a non-abelian simple group. If $\mu (G) = \mu
        (S)$, then $G \cong S$.
    \end{conj}

    \section{Some Lemmas}
    In this section, the letter $G$ always denotes a finite
group. Moreover, the letter $r$ always denotes a prime, and $r'$ is
the set of all primes distinct from $r$. For a positive integer $n$
and a set $\pi$ of primes, write $n_{\pi}$ for the $\pi$-part of
$n$.
    \begin{lemma}(\cite[Lemma 2.2]{IKS})
        If $N \unlhd G$, then $\mu_r (G) \geq  \mu_r (G/N) + \mu_r
        (N)/|G:N|$.
    \end{lemma}

    \begin{lemma}
        Let $N$ be the largest normal $r'$-subgroup of $G$.
        Then $\mu_r(G) =\mu_r (G/N)$.
    \end{lemma}
 The following lemma is a long-term conjecture of Frobenius, which
 was confirmed by
Iiyori and Yamaki in 1991 (see \cite{IY}) by using the
classification of finite simple groups.
    \begin{lemma}
Assume that the equation $x^n = 1$ in $G$ has precisely $n$
solutions. Then the set $\{x\in G|x^n=1\}$ of all these solutions is
a normal subgroup of $G$.
    \end{lemma}

    \begin{lemma}
        Let $G$ be a finite group, and $R$ be a Sylow $r$-subgroup of $G$. Then there exists an integer $t$, coprime to $r$, such that
 $\mu_r(G) = t/|R|$.
    \end{lemma}

    \begin{proof}
        Let $|G|_r$ and $|G|_{r'}$ denote the $r$-part and the $r'$-part of
$|G|$, respectively.  The number of solutions of equation $x^{|G|_{r'}} =1$ is the same as that of the $r$-regular elements in $G$.
        Therefore, by Frobenius' theorem (see \cite{FR}), the number of solutions is a multiple of $|G|_{r'}$, which implies that the number
 of $r$-regular elements $R_r(G)$ in $G$ equals $k|G|_{r'}$, where
 $k$ is a positive integer.

        Next denote by $S_r(G)$ the number of $r$-singular elements  in $G$, we have
        $\mu_r(G)={S_r(G)}/|G|={|G|^{-1}}{(|G| - {R_r}(G))}=1 -
        {|G|^{-1}}{k|G|_{r'}}=1-{k}{|G|_r^{-1}}={|R|^{-1}}{(|R|-k)},$ where
        $R$ is a Sylow $r$-subgroup of $G$. Let $t = |R| - k$, and we will prove that $t$ and $r$ are coprime, or equivalently, that the number of $r$-regular elements and $r$ are coprime. Now, let $G$ act by conjugation on the set $\Omega_r$ of $r$-regular
        elements in $G$, and let the orbit of an element $x\in \Omega_r$ be denoted by $O_x$. Note that $O_x =\{x\}$ if and only if $C_R(x)=R$. Thus, the number of $r$-regular elements is given by
        $$R_r(G) = |C_{\Omega_r}(R)|+\sum_{x\in\Omega_r-C_{\Omega_r}(C)}{|C:{C_R}(x)|},
        $$
        where $C$ is the stabilizer of $R$ in $G$. By Schur-Zassenhaus
  theorem(see e.g. 3.3.1 of \cite{KS}), we know that ${C_G}(R) = Z(R) \times H$, where $|H|$ and $r$ are coprime. Therefore, we have $C_{\Omega_r}(R) =
        {\Omega_r} \cap {C_G}(R) = H$. Hence, we get ${R_r}(G) \equiv |H|
        \bmod r$, which implies that the number of $r$-regular elements and  $r$ are coprime.     \end{proof}

  Lemma 2.4 directly proves the following remark.
    \begin{rem}
       If $\mu (G)  = \mu (H)$, then
     $|G|=|H|$.
    \end{rem}

    \begin{lemma}
        Let $G$ be a finite group, and $r$ a prime divisor of $|G|$.
        Let $x_1$, $x_2$, $\cdots$, $x_k$ be representatives of $G$-conjugacy classes of nontrivial $r$-elements in $G$. Then
        $\mu_r(G)=\sum\limits_{i = 1}^k(1-\mu_r(C_G(x_i))$.
    \end{lemma}

    \begin{proof}Note that every $r$-singular element $x$ can be written uniquely  as $x=x_rx_{r'}$ with $[x_r,x_{r'}]=1$, where $x_r$ and $x_{r'}$ are $r$- and $r'$-part of $x$, respectively (see, e.g. p.17 of \cite{PD}).  Using the same notations as Lemma 2.4, denoted by $S_r(G)$ and $R_r(G)$ the number of $r$-singular and $r$-regular elements, respectively.
    So the number of
        $r$-singular elements in $G$ is
        $S_r(G)=\sum\limits_{i=1}^k|G:C_G(x_i)|\cdot
        R_r(C_G(x_i)),
        $
        this implies that the proportion of $r$-singular elements in $G$ is
        \begin{align*}
        \mu_r(G)& =S_r(G)/|G|={|G|^{-1}}{\sum\limits_{i=
                1}^k|G:C_G(x_i)|\cdot R_r(C_G(x_i))}\\&
        =\sum\limits_{i=1}^k R_r(C_G(x_i))/{|C_G(x_i)|}
        =\sum\limits_{i=1}^k (1-\mu_r(C_G(x_i))).
        \end{align*}
        Thus $\mu_r(G)=\sum\limits_{i=1}^k (1-\mu _r(C_G(x_i)))$.
    \end{proof}

    \begin{lemma} Let $\mu_r(G)={t}/{|R|}$ as in Lemma $2.4$ and $N$ a normal subgroup of $G$.
        If  $t<r$,
        then $r\nmid |N|$ or $r\nmid|G/N|$.
    \end{lemma}

    \begin{proof}
        Suppose $r$
        divides the order of $N$. Now  assume that $r$ divides the order
        of $G/N$. Then, we can deduce that $\mu _r(G/N)\geq
        {r}/{{|R|}}$. Observe that Lemma 2.1 gives $\mu_r(G)\geq
        \mu_r(G/N)+\mu_r(N)/|G:N|$, and so
        $\mu_r(G)\geq  {r}/{|R|}$, which contradicts our
        hypothesis that $\mu_r(G)<{r}/{|R|}$. It follows that  $r\nmid |G/N|$.
    \end{proof}

    \begin{lemma} Let $p$ be a prime and $q=p^f$. Set  $r\in \pi
        (\PSL(2,q))$. Then
        $\mu_p(\PSL(2,q))$ $={1}/{q}$ if $p=2$, and $\mu_p(\PSL(2,q))=
        {2}/{q}$ if $p$ is odd. Moreover, if $r=2$ and $2|{(q\pm
            1)}/{2}$, then $\mu_r(\PSL(2,q))={1}/{2}-{1}/{(q\pm 1)_r}$; if
        $r\neq 2$ and $r|{(q\pm 1)}/{\gcd(2,q-1)}$, then
        $\mu_r(\PSL(2,q))={1}/{2}-{1}/({2(q\pm 1)_r})$.

    \end{lemma}

    \begin{proof}
        By theorems \cite[Theorem 8.2, 8.3, 8.4, 8.5, Chap.2]{H}, we know that the number
        of $p$-singular elements of $\PSL(2,q)$ is equal to the number of elements of
        order $p$. Moreover, it is $q^2-1$. Hence,
        $\mu_p(\PSL(2,q))={1}/{q}$ if $p=2$, and $\mu_p(\PSL(2,q))=
        {2}/{q}$ if $p$ is odd. Next we consider the cases of values of
        $\mu$ for distinct from the characteristic $p$.
        Similarly,
        Theorem 8.4 in \cite{H} asserts
        that if  $d \,(\neq 1)$ is a divisor of ${(q\pm 1)}/{\gcd(2,q-1)}$,
        then the number $S_d$ of elements of order $d$ in $\PSL(2,q)$ equals
        ${q(q\mp 1)\varphi(d)}/{2}$, where $\varphi$ is Euler's totient function.   Suppose that $r$ divides
        ${(q\pm 1)}/{\gcd(2,q-1)}$, and let  $(q\pm 1)/{\gcd(2,q-1)}={r^t}\cdot m$ with $\gcd(r,m)=1$. Hence the number of
        $r$-singular elements equals 
\begin{align*}
        \sum_{r\mid d\mid
            r^tm}S_d&=\sum_{d\mid r^tm}S_d-\sum_{r\nmid d\mid
            r^tm}S_d=\sum_{d\mid r^tm}S_d-\sum_{d\mid m}S_d\\&
        ={q(q\mp 1)}(\sum_{d\mid
            r^tm}\varphi(d)-\sum_{d\mid m} \varphi(d))/2={q(q\mp
            1)}(r^tm-m)/2\\&
            ={q(q\mp
            1)}({(q\pm 1)}/{\gcd(2,q-1)}-{({(q\pm
                1)}/{\gcd(2,q-1)})_{r'}})/2\\&
            =|PSL(2,q)|({1}/{2}-{1}/{(2({(q\pm
                1)}/{\gcd(2,q-1)})_r})).
        \end{align*}
It follows that
        $\mu_r(\PSL(2,q))={1}/{2}-{1}/{(2({(q\pm
                1)}/{\gcd(2,q-1)})_r)}.$
    \end{proof}

    \begin{lemma} Suppose that $\mu(G)=\mu(\PSL(2,q))$ with $q\geq 4$.
        Then $G$ is perfect.
    \end{lemma}

    \begin{proof}
        We assume that $G$ is not perfect, and let $r$ be a prime
        divisor of $|G/G'|$. By applying Lemma 2.1 we have $\mu_r(G)\geq
        \mu_r(G/G')\geq \mu_r(Z_r)=1-{1}/{r}\geq {1}/{2}$, which contradicts the
        fact of Lemma 2.8 that $\mu_r(\PSL(2,q))<{1}/{2}$ for all $r$. Hence,
        $G$ is perfect. \end{proof}

    \begin{lemma}(\cite[Lemma 1]{WJH})
        Let $S$ be a non-abelian simple group with an abelian Sylow
        $2$-subgroup.  Then $S$ is one of the following groups:
        \begin{enumerate}
            \item $\PSL(2,q)$ where $q>3$ and $q \equiv 3,5\bmod\,8$ or $q=2^m$ for $m\geq
            2$,
            \item $J_1$,
            \item ${}^2G_2(3^{2t+1})$ where $t\geq  1$.
        \end{enumerate}
    \end{lemma}

 \begin{lemma}
    Let $n$ be a positive integer. If $n\ge 6$, then $n!>2(n + 1)^{{n}/{2}}$.
\end{lemma}

\begin{proof} If $n=6$,  the inequality $n!>2(n + 1)^{{n}/{2}}$
holds by checking directly.  Now, assuming that the inequality is valid for
$n=k\geq 6$, we will proceed by induction to hold for $n=k+1$. First, we claim that
$$(k+1)^{k+2}>(k+2)^{k+1}\eqno (2.1)$$ for $k\geq 2$. In fact,
using the well-known equality
$\lim_{k\rightarrow\infty}(1+{1}/{(k+1)})^{k+1}=e$, where $e$ is
the base of the natural logarithm, it follows that $(1+{1}/{(k+1}))^{k+1}<e$.
Thus,  (2.1) holds for  $k\geq 2$.
 So,  by the induction and  (2.1), we have
$(k+1)!=k!(k+1)>2(k+1)^{{k}/{2}}(k+1)>2(k+2)^{{(k+1)}/{2}}$.
Thus it holds for $n=k+1.$
\end{proof}

    \section{Proof of Theorem}

    Let $p$ be the characteristic of $\PSL(2,q)$ and $q=p^f$.  By Remark
    2.5, if  $\mu (G) = \mu (\PSL(2,q))$, then
    $|G|=|\PSL(2,q)|$. Denote by $N$ the largest $p'$-normal subgroup in
    $G$. By Lemma 2.2 and 2.8, we have $\mu_p (G/N)={1}/{q}$ if $p=2$, and
    $\mu_p (G/N)={2}/{q}$ if $p$ is odd.  Hence, the orders of minimal
    normal subgroups of $G/N$ must necessarily have a factor of $p$.
    Moreover, if $G/N$ has two different minimal normal subgroups
    $N_0/N$ and $N_1/N$, then both orders of $G/N_0$ and $N_0/N$ have
    the divisor $p$, which contradicts Lemma 2.7. So $G/N$ has unique
    minimal normal subgroup, say $N_0/N$.

    \textbf{Claim 1.} \textit{$G/N$ is a simple group.}

    Since $\mu_p(G/N)={1}/{q}$ or ${2}/{q}$, by applying Lemma 2.6, there
    exist at most
     two conjugate classes of non-trivial $p$-elements in
    $G/N$.

    \textbf{Case I.}  \textit{$G/N$ has one conjugate class of non-trivial $p$-elements.}

    If $N_0/N$ is non-solvable, say $S\times \cdots\times S$ with $l$
    copies, and $S$ is a non-abelian simple group, then $S^l\leq G/N\leq
    \Aut(S^l)$. We choose an element $x$ of order $p$ in $S$. Since $\Aut(S^l)\cong \Aut(S)\wr Sym(l)$,  the elements $(x,1,...,1)$ and $(x,x,...,x)$ of $G/N$ are not $\Aut(S^l)$-conjugate if $l\geq 2$.
    So $l=1$, and then $G/N$ is an almost simple group. It is well-known
    that the outer automorphism group of a finite simple group is
    solvable (this is Schreier's Conjecture, see e.g. p.151 of \cite{KS}), by applying Lemma 2.10, we have that $G/N$ is simple.

    If $N_0/N$ is solvable, let $Z_p^l$, by Lemma 2.7, then $N_0/N$ is a Sylow $p$-subgroup of $G/N$. Applying Schur-Zassenhaus Theorem, $G/N$ has a $p$-complement $\overline{H}$. Now we assume that $\mu_p(G/N)={1}/{q}$. Certainly, in this case $p=2$. By Lemma 2.1,   $\mu_p(G/N)\geq \mu_p((G/N)/Core_{G/N}(\overline{H}))\geq
    \mu_p(Sym({|G/N:\overline{H}|}))=\mu_p(Sym({q}))\geq {1}/{q},$ and then
    $\mu_p((G/N)/Core_{G/N}(\overline{H}))=\mu_p(S_{q})={1}/{q}.$ Since
    $G/N$ has unique minimal normal subgroup $N_0/N$, it follows $Core_{G/N}(\overline{H})=1$. By applying a theorem of Isaacs, Kantor and Spaltenstein (see \cite{IKS}), we get that $G/N$ is sharply
    2-transitive. In the 1930’s Zassenhaus \cite{Z} classified the finite sharply 2-transitive groups based on near-fields.  Moreover, the structure of finite near-fields were researched by Ellers and Karzel (see Satz 1 of \cite{EK}), it leads $G/N$ is solvable, which contradicts the fact that $G$ is perfect in Lemma 2.10.

    Next, we consider the case of $\mu_p(G/N)={2}/{q}$. Let $x$ be an element of order $p$ in $N_0/N$. By Lemma 2.6,
    $\mu_p(G/N)=1-\mu_p(C_{G/N}(x))$, and then $\mu_p(C_{G/N}(x))=1-{2}/{q}$. Note that $C_{G/N}(x)$ still has a $p$-complement, say $\overline{H}_1$, which is a subgroup of $\overline{H}$. Then $C_{G/N}(x)\cong Z_p^f\rtimes \overline{H}_1.$ Hence the number of $p$-regular elements in $C_{G/N}(x)$ is $2|\overline{H}_1|$. We can let the number of $p'$-Hall subgroups in $C_{G/N}(x)$ be $p^t$ with $t\geq 1$. Otherwise,  $t=0$ leads to $C_{G/N}(x)\cong Z_p^f\rtimes \overline{H}_1$, hence $\mu_p(C_{G/N}(x))=1-{1}/{q}$, a contradiction. So number of $p$-regular
    elements is not less than $p^t|\overline{H}_1-M|+|M|$, where $M$ is a maximal subgroup of $\overline{H}_1$. Furthermore,
    $p^t|\overline{H}_1-M|+|M|=|M|+ p^t(|\overline{H}_1|-|M|)\geq
    {|\overline{H}_1|}/{2}+p^t(|\overline{H}_1|-{|\overline{H}_1|}/{2})
    $. If $p^t>3$, then $p^t|\overline{H}_1-M|+|M|>2|\overline{H}_1|$, a
    contradiction. Thus we have $p^t=3$ since $p$ is odd. So if
    the number of $p$-regular elements in $C_{G/N}(x)$ is
    $2|\overline{H}_1|$, then $p=3$ and $C_{G/N}(x)\cong (Z_p^f\times
    \overline{H}_2)\rtimes Z_2$ with $\overline{H}_2\leq
    \overline{H}_1$. It follows that $\overline{H}_2\subseteq
    C_{G/N}(N_0/N)$, applying $N_0/N$ is the unique minimal normal subgroup
    of $G/N$, we have $\overline{H}_2=1$. Thus $C_{G/N}(x)\cong
    Z_3^f\rtimes Z_2$ and $|Z_3^f:N_{Z_3^f}(Z_2)|=3$, and then
    $C_{G/N}(x)\cong Z_3^{f-1}\times Sym(3)$.  Since all elements of order
    $3$ are conjugate, we have $3^f-1=|G/N:C_{G/N}(x)|$, and then
    $|G/N|=2\cdot 3^f(3^f-1)$. So $|G/N_0|=2(3^f-1)$ and
    $|N|={(3^f+1)}/{4}$. This implies $8\nmid |G/N_0|$.
    Since $G$ is perfect, we have $G/N_0$ is non-solvable. Moreover,
    $G/N_0$ has a composition factor whose Sylow 2-subgroup has order 4.
    Using Walter's classification (see Lemma 2.10), we get the order of such factor
    has a prime divisor 3, which contradicts $3\nmid |G/N_0|$.

    \textbf{Case II.}  \textit{$G/N$ has two conjugate classes of
        non-trivial $p$-elements.}

    In this case it is necessary to  $\mu_p(G/N) = {2}/{q}$. As per
    Lemma 2.6, there exist two conjugacy classes of  non-trivial $p$-elements, given
    by
    $$
    \frac{2}{q}=\sum\limits_{i=1}^2{[1-{\mu_
            p}({C_{G/N}}({x_i}))]},\eqno (3.1)$$
    where ${x_1, x_2}$ are representatives of $G$-conjugacy classes of $p$- elements $\neq 1$ in $G$. Then ${\mu_ p}({C_{G/N}}({x_i})) = 1 - {1}/{q}$.  First
    we assume that $N_0/N \cong Z_p \times \cdots \times Z_p$. Then
    similarly we can express $G/N$ as a semidirect product of $N_0/N$
    and a complement subgroup $\overline H$, such that $G/N \cong
    {Z_p}^f\rtimes \overline H$. Consequently, ${Z_p}^f$ is normal Sylow
    $p$-subgroup of $G/N$.
    By Lemma 2.3,  We  conclude that ${C_{G/N}}({x_i})$ is a $p$-nilpotent group,
    and that ${C_{G/N}}({x_i}) \cong {Z_p}^f \times \overline
    {H}_i$.
    Hence, $\langle {{{{N_0}} / N},\overline {H}_1 ,\overline {H}_2 } \rangle
    \subseteq {C_{G/N}}({N_0}/N) \subseteq {C_{G/N}}({x_i}).$ It follows
    that ${C_{G/N}}({N_0}/N) = {C_{G/N}}({x_1}) = {C_{G/N}}({x_2}) \cong
    {Z_p}^f\times \overline {H}_1. $ Since $\overline {H}_1$  is a
    characteristic subgroup of ${C_{G/N}}({N_0}/N)$,  $\overline
    {H}_1$ is normal in $G/N$, and then $\overline {H}_1=1$. So $\overline H$ acts fixed-point freely by conjugation on
    $N_0/N$, then $G/N$ is a Frobenius group. By Lemma 2.9, we know
    $G/N$ is a non-solvable Frobenius group.
    Referring to \cite[Section 3.4]{DM},
    there exists a subgroup $\overline{H}_0\leq \overline{H}$ such that
    $\overline{H}_0\cong Z\times \SL(2,5)$, with
    $\big|\overline{H}:\overline{H}_0\big|\leq 2$. Since $G$ is perfect
    and $G/N$ has unique minimal normal subgroup, we can get $G/N\cong
    Z_p^f\rtimes \SL(2,5)$. But it is easy to compute
    $\mu_2(\SL(2,5))={5}/{8}$, which contradicts
    $\mu_2(\SL(2,5))=\mu_2(G/N_0)\leq \mu_2(G)<{1}/{2}$.

    Secondly, assume that ${N_0}/N$ is non-solvable, as in Case I,
    let ${N_0}/N\cong S^l$. Hence, we have $p\big|{|S|}$. Moreover,
     $S^l\leq G/N\leq
    Aut(S^l)$. The same as Case I, since $G/N$ has two conjugacy classes of
        non-trivial $p$-elements, it leads $l\leq 2$. This implies $(G/N)/S^l$ is solvable.
By Lemma 2.9,  $G$ is perfect, so that $G/N\cong S$ or
$S^2$. Now suppose that $G/N\cong S^2$. Then
$\mu_p(G/N)=1-(1-\mu_p(S))^2$, it follows $1-\mu_p(S)=\sqrt{1-2/q}$,
which is irrational, a contradiction. Thus $G/N$ is simple.

  \textbf{Claim 2.} \textit{The Sylow $p$-subgroups of $G/N$ are abelian.}

Assume that there exists non-trivial $p$-element $x$ in $G$ such
that $|C_G(x)|_p<|P|$. Then $\mu_p(G)=\sum_{x\in
X}(1-\mu_p(C_G(x)))>p/|P|$, which contradicts our hypothesis that
$\mu_p(G)=1/q$ or $2/q$. So every $p$-element of $G$ is $p$-central.
By Theorem B of \cite{NST}, and  $G/N$ is simple, we have Sylow
$p$-subgroup of $G/N$ is abelian or $G/N$ is isomorphic to Ru, $J_4$
or $^2F_4(2^{2m+1})'$ (with $9\nmid (2^{2m+1}+1)$) and $p=3$; or
$G/N\cong Th$ and $p=5$. With the aid of the ATLAS (see \cite{CCN}),
we know $\mu_3(Ru)=8/27, \mu_3(J_4)=8/27, \mu_3(^2F_4(2)')=7/27$ and
$\mu_5(Th)=24/125$. By Table 2 of \cite{GS},
$\mu_3(^2F_4(2^{2m+1}))=1-(89/144+1/4(2^{2m+1}+1)_3+1/48(2^{2m+1}+1)^2_3)$.
It is easy to check all values of $\mu$ above are not equal to $1/q$
or $2/q$.  Thus the Sylow $p$-subgroups of $G/N$ are abelian.

    \textbf{Claim 3.}  \textit{$G\cong \PSL(2,q)$ for even $q>2$.}

    Let the exponent of Sylow 2-subgroup of $G$ be $2^e$. Since the
    number $S_{2^e}(G)$ of elements in $G$ whose order is divisible by
    $2^e$ has divisor $|G|_{2'}$ (see \cite[Theorem 3]{W}), we have
    $(2^f+1)(2^f-1)\big|S_{2^e}(G)$. On the other hand, since $\varphi
    (2^e)\big|S_{2^e}(G)$, it follows  $2^{e-1}(2^f+1)(2^f-1)\left|
    {S_{2^e}(G)}\right.$. By applying Lemma 2.9, we get the number
    $S_2(G)$ of 2-singular elements of $G$ equals $(2^f+1)(2^f-1)$.
    However,
    $S_2(G)\ge S_{2^e}(G)\ge 2^{e-1}(2^f+1)(2^f-1)$,
    hence $e=1$, and so $G$ has an elementary abelian Sylow 2-subgroup.
    Certainly, $G/N$ has an elementary Sylow 2-subgroup too. Since $G/N$
    is isomorphic to a non-abelian simple group again, we can use
      Walter's classification (see Lemma 2.10) to split the proof into the following
    three cases.

    \textbf{Case I.}  \textit{$G/N\cong \PSL(2,q_1)$, where $q_1>3$ and
        $q_1 \equiv 3,5\mod\,8$ or $q_1=2^m$ with $m\geq  2$.}

    When $q_1 \equiv 3,5\mod\,8$, we have $8\nmid {(q_1 \pm1)_2}$, and
    then ${(q_1 \pm1)_2}=4$. By applying Lemma 2.4 it yields ${\mu
        _2}(\PSL(2,q_1))= {1}/{2}-{1}/{(q_1 \pm1)_2}={1}/{4}$.
    Moreover,  ${\mu _2}(G)={1}/{2^f}$ and ${\mu _2}(G)={\mu
        _2}(G{\rm{/}}N)$, and therefore, $f\leq 2$. Thus, we obtain ${\mu
    }(G)={\mu }(\PSL(2,4))$, and then $|G|=60$. Since $G$ is perfect,  $G\cong \PSL(2,4)$. Next if $q_1=2^m$ where $m\geq 2$,
     observe that ${\mu _2}(\PSL(2,q_1)) ={1}/{2^m}$ by Lemma 2.8.
    However, since ${\mu _2}(G) ={1}/{2^f}$ and
    ${\mu_2}(G)={\mu_2}(G/N)$, it follows that $m=f$. As
    $|G/N|=|\PSL(2,q_1)|=2^m(2^m+1)(2^m-1)$ and
    $|G|=|\PSL(2,p^f)|=2^f(2^f+1)(2^f-1)$, we obtain $|G|=|G/N|$, which
    implies $G\cong \PSL(2,2^f)$.

    \textbf{Case II.} $G/N\cong J_1$.

    Referring to Table of ATLAS \cite{CCN},
    ${\mu_2}(J_1)={3}/{8}$. However,
    ${\mu_2}(G)={1}/{2^f}$ and ${\mu_2}(G)={\mu_2}(G/N)$, so ${\mu_2}(J_1)={1}/{2^f}$. This leads to a contradiction
    with ${\mu_2}(J_1)={3}/{8}$.

    \textbf{Case III.} \textit{$G/N\cong {}^2G_2(3^{2m+1})$ where $m\geq
        1$.}

    Observe that
    ${\mu_2}({}^2{G_2}({3^{2t+1}}))=1-{7}/{{12}}-{1}/{{6{{({3^{2t+1}}+1)}_2}}}=
    {5}/{{12}}-{1}/{{6{{({3^{2t+1}}+1)}_2}}}$ (see \cite[Table
    11]{GS}). As $8\nmid {(3^{2t+1}+1)_2}$, we conclude that
    ${(3^{2t+1}+1)_2}=4$, and then ${\mu _2}({}^2{G_2}({3^{2t +
            1}}))={3}/{8}$. Note that ${\mu_2}(G)={1}/{2^f}$ and
    ${\mu_2}(G)={\mu_2}(G/N)$. It forces
    ${3}/{8}={1}/{2^f}$, which is a
    contradiction.

    \textbf{Claim 4.}  \textit{$G\cong \PSL(2,q)$ for odd $q>3$.}

    First assume that $G/N$ is isomorphic to $A_n$, where $n \geq  5$.  Since Sylow
    $p$-subgroup of $G$ is abelian,  it follows that $n<p^2$. Therefore,
    $$
    |G/N|_p=|A_n|_p=p^{[{n}/{p}]+[{n}/{p^2}]+\cdots}=p^{[{n}/{p}]}
    \leq p^{{n}/{p}}. \eqno (3.2)$$ Since the Sylow $p$-subgroup of
    $G/N$ is a Sylow $p$-subgroup of $G$, it deduces from (3.2) that
    ${n}/{p}>f$, that is
    $$n>pf.\eqno (3.3)$$
    Furthermore, for $n
    \geq  6$, the inequality $n!>2(n + 1)^{{n}/{2}}$ holds by Lemma 2.12, and combining  (3.3) above, we can obtain that
    $$
    {n!}/{2}>(pf+1)^{{pf}/{2}}. \eqno (3.4)$$
    Assuming that $p \geq  7$, it is easy from (3.4) to see that ${n!}/{2}>(pf+1)^{{pf}/{2}}\geq
    p^{3f}>{p^f(p^{2f}-1)}/{2}, $ which leads  to $|G/N|>|G|$,  a
    contradiction. Assuming $p=5$ and $f\geq 2$, we observe from (3.4)
    that $
    {n!}/{2} >(5f+ 1)^{{5f}/{2}}\geq 5^{3f}> {{{5^f}({5^{2f}} - 1)}}/{2}
    $,  a contradiction. When $p = 5$ and $f=1$,  if $n\geq 6$, then
    $|G/N|={n!}/{2}>{5(5^2-1)}/{2}=|G|$, a contradiction. Next
    when $p=3$ and $f\geq 3$, by the inequality (3.4),
    $
    {n!}/{2}>(3p+1)^{{pf}/{2}}={10^{{3f}/{2}}}>{{{3^f}({3^{2f}}-1)}}/{2},
    $
    a contradiction too.

    For $p=3$ and $f=2$, the formula (3.4) deduces $n=6$, $7$ or $8$.
    When $n=6$, we have ${6!}/{2}$ dividing ${{3^f}({3^{2f}} -1)}/{2}$, yielding $G/N\cong{A_6}$.
    Moreover, since $|G|=|\PSL(2,p^f)|$, it follows that $G\cong \PSL(2,9)$. For $n=7$ or $8$, it turns out that
    ${n!}/{2}>{{{3^2}({3^{4}} - 1)}}/{2}$, a contradiction.

    Lastly we assume that  $G/N\cong A_5$. If $p=3$ and 5, then
    $\mu_3(G)=\mu_3(A_5)={1}/{3}$ and $\mu_5(G)={2}/{5}$, respectively.  But
    $\mu_p(\PSL(2,q))={2}/{q}$, then we have  $p=q=5$. Thus $G\cong
    \PSL(2,5)$.

    Next we suppose that $G/N$ is a simple group of Lie type
    $L(q_0)$.   By using  Artin invariant of
    the cyclotomic factorisation (ref. \cite{A}), the order of $L(q_0)$ has the type
$$|L(q_0)|={d^{-1}}q_0^h\prod_m\Phi_m(q_0)^{e_L(m)},\eqno (3.5)$$ \\ where $d$ is the denominator
and $h$ the exponent given for $L(q_0)$ in Table 1 and Table 2.
\begin{table}[ht]
        \caption{\bf~Cyclotomic factorisation: classical groups of Lie
            type}~~~\begin{tabular}{llll} \hline
            \quad $L$ &$ d$ &$ h$ & $ e_L(x)$ where $x$ is an integer \\[2pt]
            \hline
            $\PSL(n,q_0)$ & $\gcd(n,q_0-1)$   & $n(n-1)/2$& $~~n-1~~~$if $x=1$\\[2pt]
            &&&$~~[\frac{n}{x}]~~~$ if $x>1$ \\[2pt]\hline
            &&&$~~[\frac{n}{\rm{lcm}(2,x)}]$~ if $x\not\equiv 2(\bmod\,4)$\\[2pt]
            $\PSU(n,q_0)$&$\gcd(n,q_0+1)$&$n(n-1)/2$& ~~$n-1$ ~ if $x=2$\\[2pt]
            &&&$~~[\frac{2n}{x}]$ ~ if $2<x \equiv 2(\bmod\,4)$\\[2pt]\hline
            $\PSp(2n,q_0)$ &$\gcd(2,q_0-1)$ &$n^2$&$[\frac{2n}{\rm{lcm}(2,x)}]$\\
            $\OO(2n+1,q_0)$ &&&\\[4pt]\hline

            $\OO^+(2n,q_0)$&$\gcd(4,q_0^n-1)$&$n(n-1)$&$[\frac{2n}{\rm{lcm}(2,x)}]$ ~if
            $x|n$ or $x\nmid 2n$\\[4pt]
            &&& $\frac{2n}{x}-1$~ if $x\nmid n$ and $x| 2n$
            \\[2pt]\hline            $\OO^-(2n,q_0)$&$\gcd(4,q_0^n+1)$&$n(n-1)$&$[\frac{2n}{\rm{lcm}(2,x)}]-1$
            ~if
            $x|n$ \\[2pt]
            &&& $[\frac{2n}{\rm{lcm}(2,x)}]$ ~if $x\nmid n$
            \\[2pt]\hline
        \end{tabular}
    \end{table}
    \begin{table}[ht]
        \caption{\bf~Cyclotomic factorisation: exceptional groups of Lie
            type}~~~\begin{tabular}{lllllllllll}
            \hline
            $m$& $^2B_2$ &$^3D_4$ &$ G_2$ & $ ^2G_2$
            &$F_4$&$^2F_4$&$E_6$&$^2E_6$&$E_7$&$E_8$
            \\\hline
            1&1&2&2&1&4&2&6&4&7&8\\\hline 2&&2&2&1&4&2&4&6&7&8\\\hline
            3&&2&1&&2&&3&2&3&4\\\hline 4&1&&&&2&2&2&2&2&4\\\hline
            5&&&&&&&1&&1&2\\\hline 6&&2&1&1&2&1&2&3&3&4\\\hline
            7&&&&&&&&&1&1\\\hline 8&&&&&1&&1&1&1&2\\\hline
            9&&&&&&&1&&1&1\\\hline 10&&&&&&&&1&1&2\\\hline
            12&&1&&&1&1&1&1&1&2\\\hline 14&&&&&&&&&1&1\\\hline
            15&&&&&&&&&&1\\\hline 18&&&&&&&&1&1&1\\\hline 20&&&&&&&&&&1\\\hline
            24&&&&&&&&&&1\\\hline 30&&&&&&&&&&1\\\hline
            $d$&1&1&1&1&1&1&$\gcd(3,q-1)$&$\gcd(3,q+1)$&$\gcd(2,q-1)$&1\\\hline
            $h$&2&12&6&3&24&12&36&36&63&120\\\hline
        \end{tabular}
\rm{Note that for Table 2, the numbers in the row headings are those integers $m$ for which a cyclotomic polynomial $\Phi_m(q)$ enters into the cyclotomic factorisation in
terms of $q$ of one of the exceptional groups $L(q)$ of Lie type according to the formula
(3.5). The types $L$ are the headings, and the entry in the m-row and the $L$-column is
$e_L(m)$ if non-zero and blank otherwise. }
    \end{table}

By Claim 2,   $L(q_0)$ has an abelian Sylow
    $p$-subgroup. Next we cite
     the classification of simple groups with
    some abelian Sylow subgroups (see the items from (2) to (4) of
    \cite{SZ}).

  \begin{lemma} Let $L(q_0)$ be a finite simple group of Lie type over the Galois field $GF(q_0)$. Suppose that $L(q_0)$ has an abelian Sylow $p$-subgroup and $p$ an odd prime. Then $L(q_0)$ is one of the following groups:
    \begin{enumerate}
        \item a linear simple group $\PSL(2,q_0)$ with $p|q_0$,
        \item  $\PSL(3,q_0)$, $q_0\equiv 4, 7(\bmod\,9)$ with $p = 3$,
        \item  $\PSU(3,q_0 )$, $2 < q_0 \equiv 2, 5(\bmod\,9)$ with $p = 3$,
        \item $L(q_0)$ satisfied  $e_L(mr_m) = 0$ except the above groups of $(1)$ and $(2)$,
        where $p = r_m$ is a primitive prime of $q_0^m -1$, the function $e_L(x)$ is defined in Table $1$ and Table
        $2$.
    \end{enumerate}
\end{lemma}
Next we divide into four cases.

    Case I. $p|q_0$ and $G/N\cong \PSL(2,q_0)$. Since
    $\mu_p(G)=\mu_p(\PSL(2,q))={2}/{q}$ and
    $\mu_p(G/N)=\mu_p(\PSL(2,q_0))={2}/{q_0}$, we have $q=q_0$, and
    thus $G\cong \PSL(2,q)$.

    Case II. $p=3$ and $G/N\cong \PSL(3,q_0)$, where $q_0\equiv
    4,7(\bmod\,9)$. Note that $|\PSL(3,q_0)|_3=9$. So $q=9$, and then
    $|G|=|\PSL(2,9)|=360$. It follows that
    $|G/N|=|\PSL(3,q_0)|={q_0}^3({q_0}^2-1)({q_0}^3-1)/3$ divides
    360, a contradiction.

    Case III. $p=3$ and $G/N\cong \PSU(3,q_0)$,  where $2<q\equiv
    2,5(\bmod\,9)$. Since $|\PSU(3,q_0)|_3=9$, we have $q=9$, and then
    $|G|=|\PSL(2,9)|=360$. It follows that
    $|G/N|=|\PSU(3,q_0)|={q_0}^3({q_0}^2+1)({q_0}^3-1)/3$ divides
    360, a contradiction.

    Case IV. $G/N\cong L(q_0)$  with ${e}_L(mr_m)=0$ neglecting the groups of Case II and III above, where
    $r_m$ is a primitive prime of ${q_0}^m-1$, the function $e_L(x)$ is
    defined in Table 1 and Table 2.

    Since $|G/N|_p=|L(q_0)|_p$, we have
    $$p^f=(\Phi_m(q_0))_p^{e_L(m)}\eqno
    (3.6)$$
    (see \cite [Lemma 3.5] {SZ}). By $|L(q_0)|\big||G|$, we have
    $|L(q_0)|_{p'}\big|{(p^{2f}-1)}/{2}$.
    Thus
    $${d^{-1}}{q_0}^h\prod_{t\neq m}\Phi_t(q_0)^{e_L(t)}\big|{(p^{2f}-1)}/{2}.\eqno (3.7)$$
    Note that ${\Phi_m}({q_0}) < 4{q_0}^{\varphi (m)}$ (see \cite[Lemma
    2.1]{GLN}), thus from (3.6) we have $
    {(p^{2f}-1)}/{2}<{\Phi_m}(q_0)^{2e_L(m)}<(4{q_0}^{\varphi(m)})^{2e_L(m)}$.
    Combining (3.7), it follows that ${d^{-1}}{q_0}^h\prod_{t\neq m}\Phi_t(q_0)^{e_L(t)} $ $< (4{q_0}^{\varphi (m)})^{2e_L(m) } $.
    It is easy to check from Table 1 and  2 that
    ${d^{-1}}q^h\prod_{t\neq m}\Phi_t(q_0)^{e_L(t)}$ $>{q_0}^{h+1}$.
    So
    ${q_0}^{h+1} < (4{q_0}^{\varphi (m)})^{2e_L(m)}<{q_0}^{2(m+1)e_L(m)}$,
    and then
    $$h+1  < 2(m+1)e_L(m).\eqno (3.8)$$

    In the sequel, we discuss case by case.

    Subcase 4.1. $G/N\cong \PSL(n,q_0)$, where $n\geq 2$.

    Using Table 1, we have $e_L(x)=[{n}/{x}]$~ if $x>1$;
    $e_L(x)=n-1$ if $x=1$. When $m>1$,  by the inequality (3.8), it follows that
    ${n(n-1)}/{2}+1 < 2(m+1)[{n}/{m}] \leq 3n$, thus $n<7$.
    Since $2\leq m\leq n$,  the possible pair
    $(n,m)$ is one of $(2,2)$, $(3,3)$, $(3,2)$, $(4,2)$, $(4,3)$, $(4,4)$, $(5,2)$,
    $(5,3)$, $(5,4)$, $(5,5)$, $(6,2)$, $(6,3)$, $(6,4)$, $(6,5)$ and $(6,6)$.
    We will check that  $|\PSL(n,q_0)|>|G|$ in these cases.

    First we consider the case $e_L(m)=1$, i.e. $(n,m) = (2,2)$, $(3,2)$, $(3,3)$, $(4,3)$, $(4,4)$,
    $(5,3)$, $(5,4)$, $(5,5)$,  $(6,4)$, $(6,5)$ and $(6,6)$.
    Now if $(n,m)\neq (2,2)$, $(3,2)$, $(5,5)$ and $(6,5)$,
    then  $|G/N|=|\PSL(n,q_0)|\ge|\PSL(3,q_0)|={q_0}^{3}({q_0}^2 - 1)({q_0}^3- 1)/\gcd(3,{q_0} - 1)$.
    By (3.5), it's evident that ${p^f} = ({\Phi _m}({q_0}))_p\le{q_0}^2+{q_0}+1$.
    So $|G|=|\PSL(2,p^f)|={p^f(p^{2f}-1)}/{2}
    \le{{({q_0}^2+{q_0}+1)(({q_0}^2+{q_0}+1)^2-1)}/{2}}$.
    It's easy to check if $q_0>2$,  then $|G/N|=|\PSL(n,q_0)|>|G|$,
    a contradiction. Now if $q_0=2$, then we also have $|G/N|>|G|$ except $n=3$.
    Certainly, when  $q_0=2$ and
     $n=3$,  $G/N\cong \PSL(3,2)$ and $|G|=|\PSL(2,7)|$, and so $G\cong \PSL(2,7)$.
     If $(n,m)=(2,2)$,
    $(3,2)$,
  then $|G/N|=|\PSL(n,q_0)|\ge {q_0}({q_0}^2 -
    1)/\gcd(2,{q_0} - 1)$.
    Since ${p^f} = ({q_0} + 1)_p$, it leads to $|G|=|\PSL(2,p^f)| $ $<{({q_0}+1)^{3}}$.
    It turns out  $|G/N|=|\PSL(n,q_0)|>|G|$, a contradiction.
    If $(n,m)=(6,5)$ and $(5,5)$,
    then $|G/N|=|\PSL(n,q_0)|\ge{q_0}^{10}({q_0}^2 - 1)({q_0}^3 - 1)({q_0}^4 - 1)({q_0}^5 - 1)/\gcd(5,{q_0} - 1)$,
    and  ${p^f} =({\Phi _5}({q_0}))_p$,
    and so $|G|=|\PSL(2,p^f)| <{({q_0}^4+{q_0}^3+{q_0}^2+{q_0}+1)^{3}}$.
     It's obvious that $|G/N|=|\PSL(n,q_0)|>|G|$, a contradiction.

    Next we prove  the case $e_L(m)=2$, i.e. $(n,m)=(4,2)$, $(5,2)$ and $(6,3)$.
    Note that $|G/N|=|\PSL(n,q_0)|\ge{q_0}^{6}({q_0}^2 - 1)({q_0}^3 - 1)({q_0}^4 - 1)/\gcd(4,{q_0} - 1)$.
    It is clear that ${p^f} =({q_0}^2+{q_0}+1)^2_p$, and so $|G|=|\PSL(2,p^f)|
    <{({q_0}^2+{q_0}+1)^{6}}$. Moreover, since $|G/N|\leq |G|$, we have
    $q_0=2$, and then $p=7$ and $f=2$. So $|G|=|\PSL(2,49)|$, $|G/N|=|\PSL(4,2)|,
    |\PSL(5,2)|$ or $|\PSL(6,2)|$. Obviously, $|G/N|\nmid |G|$, a
    contradiction.

    Finally,  we discuss the case $e_L(m)=3$, that is $(n,m)= (6,2)$.
    It follows from (3.5) that ${p^f} = ({q_0}+1)_p^3$,
    which leads to $| G | < {p^{3f}}\le {({q_0} +1)^{9}}$.
    Since  $|G/N|=| {\PSL(6,{q_0})}| = {q_0}^{15} ({q_0}^2 - 1)\cdots({q_0}^6 - 1)/\gcd(6,{q_0} - 1)>{({q_0} +1)^{9}}$.
    Therefore, we can deduce that  $|G/N|>|G|$, a contradiction.

   If $m=1$, then
   $e_L(m)=n-1$.  We obtain from (3.5) that $p^f= ({q_0} - 1)_p^{n - 1 }$.
   Since $|G|=|\PSL(2,p^f)|$,
   we have ${({p^{2f}} - 1)}/{2} = {({({q_0} - 1)_p^{2(n - 1) }} - 1)}/{2}$. Furthermore,
   $|\PSL(n,q_0)|$ divides $|G|$, it follows
  ${q_0}^{{n(n-1)}/{2}}({q_0}+1)({q_0}^2+{q_0}+1)\cdots({q_0}^{n-1}+\cdots+1)_{p'}/{\gcd(n,{q_0}-1)}  $
   is not greater than ${(({q_0}-1)_p^{2(n - 1) } - 1)}/{2}$.
   Therefore, we obtain
   $
   {q_0}^{{n(n-1)}/{2}} \cdot {q_0}^2 \cdot {q_0}^3 \cdots {q_0}^{n - 1} < {({q_0} - 1)^{2(n - 1)}} < {q_0}^{2(n - 1)},
   $
   which yields ${n(n-1)}/{2}+{(n+1)(n-2)}/{2}<2(n - 1)$.
   Hence, we conclude that $n=2$, and so $G/N \cong \PSL(2, q_0)$.
   Note that $|G/N|={{q_0}({q_0}^2-1)}/{\gcd(2,{q_0}-1)}$ and $|G|=|\PSL(2,p^f)|={({q_0}-1)_p(({q_0}-1)_p^2-1)}/{2}$,
   it shows $|G/N|>|G|$, a contradiction.

    Subcase 4.2. $G/N \cong \PSU(n,q_0)$, where $n \geq 3$.

    In the light of Table 1,  note that $e_L(x)=[{n}/{\rm{lcm}(2,x)}]$~ if $x\not\equiv 2(\bmod\,4)$;
    $e_L(x)=n-1$ if $x=2$; $e_L(x)=[{2n}/{x}]$~ if $2<x\equiv 2(\bmod\,4)$.
    It follows from (3.7) that
    ${n(n-1)}/{2}+1<2(m+1)e_L(m)$. If $m\not\equiv 2(\bmod\,4)$, then ${n(n-1)}/{2}+1<2(m+1)[{n}/{\rm{lcm}(2,m)}]\leq 2(m+1)[{n}/{m}]\leq
    3n$, and so $n<7$. Since  $m\not\equiv 2(\bmod\,4)$,  the possible value $(n,m)$ is one of $(3,1), (4,1), (4,4), (5,1),(5,4)$, $(6,1)$, $(6,3)$, $(6,4)$.
   In the following, it will be proved that
   $|\PSU(n,{q_0})|>|G|$  for the above cases.
   In the first case, $e_L(m)=1$, i.e. $(n,m) =(3,1)$, $(4,4)$, $(5,4)$, $(6,3)$, $(6,4)$.
  This implies  $|G/N|=|\PSU(n,{q_0})|\ge{q_0}^3({q_0}^2-1)({q_0}^3+1)/{\gcd(3,{q_0}+1)}$.
    By (3.5), we have $p^f=(\Phi_m(q_0))_p\leq({q_0}^2+{q_0}+1)$,
    leading to $|G|=|\PSL(2,p^f)|< ({q_0}^2+{q_0}+1)^3$.
    It shows that $|G/N|>|G|$, a contradiction.
   In the second case,  if $e_L(m)=2$, then $(n,m)=(4,1)$ and $(5,1)$.
   It follows $|G/N|=|\PSU(n,{q_0})|\ge{q_0}^6({q_0}^2-1)({q_0}^3+1)({q_0}^4-1)/{\gcd(4,{q_0}+1)}$.
    Moreover, $p^f=(q_0-1)_p^2$,
   hence $|G|=|\PSL(2,p^f)|< ({q_0}-1)^6$.
    It's an easy calculation that $|G/N|>|G|$, a contradiction.
    In the last case, $e_L(m)=3$, that is $(n,m)=(6,1)$.
    According to (3.5), we have $p^f=(q_0-1)_p^3$, so $|G|=|\PSL(2,p^f)|< ({q_0}-1)^9$.
    Furthermore, since $|G/N|=|\PSU(6,{q_0})|={q_0}^{15}({q_0}^2-1)({q_0}^3+1)({q_0}^4-1)({q_0}^5+1)({q_0}^6-1)/{\gcd(6,{q_0}+1)}$,
   we have $|G/N|>|G|$ too,  a contradiction.

   If  $m=2$,  then $e_L(m)=n-1$.
   Since $|\PSU(n,{q_0})|\le|G|< p^{3f}$,  it follows that
   $({q_0}^2-1)({q_0}^3+1)\cdots({q_0}^n-(-1)^n)<({q_0}+1)^{3(n-1)}/{\gcd(n,{q_0}+1)}$,
   that is ${q_0}^{n(n-1)-2}<{({q_0}+1)^{3(n-1)}}$. Therefore, we conclude that $n=3$. For $n=3$,
   $|G/N|={q_0}^3({q_0}^2 - 1)({q_0}^3 + 1)/{{\gcd(3,{q_0} + 1)}} > {({q_0} + 1)^6}>|G|,$
   a contradiction.

    If $2<m\equiv 2(\bmod\,4)$, then $e_L(m)=[{2n}/{m}]$.
    We have ${n(n-1)}/{2}+1<2(m+1)[{2n}/{m}]\leq
    {14}n/3$, and so $n< 11$.
    It follows that the pairs $(n,m)$ is one of $(3,6)$, $(4,6)$, $(5,6)$, $(6,6)$, $(6,10)$, $(7,6)$,
    $(7,10)$, $(7,14)$, $(8,6)$, $(8,10)$, $(8,14)$, $(9,6)$, $(9,10)$, $(9,14)$, $(9,18)$, $(10,6)$, $(10,10)$, $(10,14)$ and $(10,18)$.
     First we discuss the case of $e_L(m)=1$, that is $(n,m)
     =(3,6)$, $(4,6)$, $(5,6)$, $(6,10)$, $(7,10)$, $(7,14)$, $(8,10)$, $(8,14)$, $(9,10)$, $(9,14)$, $(9,18)$,  $(10,14)$ and $(10,18)$.
     If $(n,m)\neq(3,6)$, $(4,6)$ and $(5,6)$, then  $p^f={(\Phi_m({q_0}))_p}\leq {q_0}^6+1$, this leads to
     $| G |=|\PSL(2,p^f)| < {p^{3f}}\le {( {q_0}^6+1)^{3}}$.
     We observe that $|G/N|=|\PSU(n,{q_0})|\ge {q_0}^{15}({q_0}^2-1)({q_0}^3+1)({q_0}^4-1)({q_0}^5+1)({q_0}^6-1)/{\gcd(6,{q_0}+1)}$.
    It is easy to check $|G/N|>|G|$, a
     contradiction. If $(n,m)=(3,6)$, $(4,6)$ and $(5,6)$, we have ${p^f} \le {q_0}^2-{q_0}+1$.
     It follows that $|G/N|=| \PSU(n,q_0)| \ge {q_0}^{3} ({q_0}^2 - 1)({q_0}^3 + 1)/{{\gcd(3,{q_0} + 1)}}>({q_0}^2-{q_0}+1)^3$.
     On the other hand, $| G |=|\PSL(2,p^f)| < {p^{3f}}\le {({q_0}^2-{q_0} +1)^{3}}$, then $|G/N|> |G|$, a contradiction.
     Next when $e_L(m)=2$, i.e. $(n,m) =(6,6)$, $(7,6)$, $(8,6)$, $(9,6)$, $(10,6)$ and $(10,10)$.
     According to (3.5),   $p^f={(\Phi_m({q_0}))_p^2} \leq ({q_0}^4-{q_0}^3+{q_0}^2-{q_0}+1)^2$, thus
     $| G |=|\PSL(2,p^f)| < {p^{3f}}\le {( {q_0}^4-{q_0}^3+{q_0}^2-{q_0}+1)^{6}}$.
     Moreover, note that $|G/N|=|\PSU(n,{q_0})|\ge {q_0}^{15}({q_0}^2-1)({q_0}^3+1)\cdots ({q_0}^6-1)/{\gcd(6,{q_0}+1)}$.
      We have $|G/N|=|\PSU(n,q_0)|>|G|$, also a contradiction.

    Subcase 4.3. $G/N\cong \PSp(2n, q_0)$, $\OO(2n+1,q_0)$, where $n \geq 2$.

    In view of Table 1, we have  $e_L(x)=[{2n}/{\rm{lcm}(2,x)}]$. If $m>2$, it follows using (3.7) that
    $n^2+1<{ 2(m+1){[{2n}/{\rm{lcm}(2,m)}]}}\leq 5n$, and so $n <5$.
    So the possible value $(n,m)$ is one of  $(2,4)$, $(3,3)$, $(3,4)$, $(3,6)$, $(4,3)$, $(4,4)$, $(4,6)$ and  $(4,8)$.
   We will prove that  $|\PSp(2n, q_0)|>|G|$ for  all cases above in the following.
   First we consider the case of $e_L(m)=1$, i.e. $(n,m) =(2,4)$, $(3,3)$, $(3,4)$, $(3,6)$, $(4,3)$, $(4,6)$ and  $(4,8)$,
  and then  $|G/N|= |\PSp(2n, q_0)|\ge {q_0}^4 ({q_0}^2 - 1)({q_0}^4 - 1)/{{\gcd(2,{q_0} - 1)}}$.
  Now if $(n,m)\neq (4,8),$
   then ${p^f} = ({\Phi_m(q_0))_p}\le{q_0}^2+{q_0}+1$ by (3.5), and so
   $| G |=|\PSL(2,p^f)| < {p^{3f}}\le {({q_0}^2+{q_0}+1)^{3}}$.
   It's easy to check  $|G/N|=|\PSp(2n, q_0)|>|G|$, a
   contradiction. If $(n,m)=(4,8)$, it implies that
   $|G/N|=| {\PSp(8,{q_0})}| = {q_0}^{16} ({q_0}^2 - 1)({q_0}^4 - 1)({q_0}^6-1)({q_0}^8-1)/{{\gcd(2,{q_0} - 1)}}$.
   Then   ${p^f} = ({q_0}^4+1)_p$ follows $| G | < {p^{3f}}\le {({q_0}^4 +1)^{3}}$.
   It is obvious that $|G/N|>({q_0}^4+1)^3>|G|$, a contradiction.
   Next $e_L(m)=2$,  that is $(n,m) =(4,4)$.
   Observe that $|G/N|= |\PSp(8, q_0)|={q_0}^{16} ({q_0}^2 - 1)({q_0}^4 - 1)({q_0}^6-1)({q_0}^8-1)/{{\gcd(2,{q_0} - 1)}}$.
   Since ${p^f} = (\Phi_m(q_0))_p^2\le({q_0}^2+1)^2$, it implies
   $| G |=|\PSL(2,p^f)| < {p^{3f}}\le {({q_0}^2+1)^{6}}$.
   Therefore, we deduce that  $|G/N|>|G|$, a contradiction.

    Finally, if  $m=1$ or $m=2$, then from $p^f\leq({q_0}+ 1)_p^n$ , we have $|G|=|\PSL(2,p^f)|<({q_0}+ 1)^{3n}$.
    Furthermore, since  $|G/N|=|\PSp(2n, q_0)|>{q_0}^{2n^2}$,  it
    is easy to check
    that $|\PSp(2n, q_0)|>({q_0}+ 1)^{3n}$, that is $|G/N|>|G|$, a contradiction.

    Subcase 4.4. $G/N\cong \OO^+(2n,q_0)$, where $n \geqslant 3$.

    Applying Table 1, we have $e_L(x)=[{2n}/{\rm{lcm}(2,x)}]$~ if  $x| n$ or $x\nmid 2n$;
    $e_L(x)={2n}/{x}-1$~ if $x\nmid n$ and $x| 2n$.
     If $m| n$ or $m\nmid 2n$, then $e_L(m)=[{2n}/{\rm{lcm}(2,m)}]$,
     where $m>2$. It follows from (3.8) that the inequality $n(n-1)+1<2(m+1)[{2n}/{\rm{lcm}(2,m)}]\leq 5n$, thus $n < 6$,
     that is the pairs $(n,m) = (3,3)$, $(3,4)$, $(4,3)$, $(4,4)$, $(4,6)$, $(5,3)$, $(5,4)$, $(5,5)$, $(5,6)$, $(5,8)$.
     In  the following, it will be shown that $|\OO^+(2n,q_0)|>|G|$ in these cases.
    In the first case, we see that $e_L(m)=1$, i.e. $(n,m) = (3,3)$, $(3,4)$, $(4,3)$,  $(4,6)$, $(5,3)$,
     $(5,5)$, $(5,6)$, $(5,8)$. It follows
     $|G/N|=| \OO^+(2n,q_0)|\geq | {\OO^+(6,{q_0})}| = {q_0}^6 ({q_0}^3-1)({q_0}^2 - 1)({q_0}^4 -1)/{{\gcd(4,{q_0}^3 - 1)}}$.
     According to (3.5), we have ${p^f} = {(\Phi_m(q_0))_p}$. Now if $(n,m)\neq (5,5),$
   then ${p^f} = ({\Phi_m(q_0))_p}\leq {q_0}^4+1$, and so
     $| G |=|\PSL(2,p^f)| < {p^{3f}}\le {({q_0}^4+1)^{3}}$.
   We conclude that  $|G/N|=|\OO^+(2n,q_0)|>|G|$, a
     contradiction. If $(n,m)=(5,5)$,
     then ${p^f} = {(\Phi_m({q_0}))_p}\leq {{(q_0}^5-1)}/{{(q_0}-1)}<{q_0}^5$, and so
      $|G/N|=| {\OO^+(10,{q_0})}|= {q_0}^{20} ({q_0}^5-1)({q_0}^2 - 1)({q_0}^4 -
      1)({q_0}^6-1)({q_0}^8-1)/{{\gcd(4,{q_0}^5 - 1)}}>{q_0}^{15}$. On the other hand,
      $|G|=|\PSL(2,p^f)|<p^{3f}<{q_0}^{15}$, then
      $|G|<{q_0}^{15}<|G/N|$, a contradiction. In the second case,  $e_L(m)=2$, that is $(n,m) = (4,4)$,
      $(5,4)$. Clearly, $|G/N|=|\OO^+(2n,q_0)|\geq {q_0}^{12}({q_0}^2 - 1)({q_0}^4 -
      1)^2({q_0}^6-1)/{\gcd(4,{q_0}^4 - 1)}$.
      We obtain from (3.5) that ${p^f} = {({q_0}^2+1)_p^2}$, which implies  $|G|=|\PSL(2,p^f)|<p^{3f}\le({q_0}^2+1)^6$.
      It implies that $|G/N|>|G|$, a contradiction.

    If $m>2$, $m\nmid n$ and $m| 2n$, then $e_L(m)={2n}/{m}-1$ .
    It follows that $n(n-1)+1<2(m+1)({2n}/{m}-1)<4n + {4n}/{3}$. Thus $n < 7$,
    that is the pairs $(n,m)$ is one of $(3,6)$, $(4,8)$, $(5,10)$, $(6,4)$ and $(6,12)$.
    When $(n,m) = (3,6)$, $(4,8)$, $(5,10)$ and $(6,12)$,
    by (3.5), ${p^f} = {\Phi_m}{({q_0})^{{2n}/{m}-1}} = 1$,  which contradicts the assumption that $p>3$.
    Next if $(n,m) = (6,4)$, then ${p^f} = {({q_0}^2+1)_p^2}$,
   it implies that $| G |=|\PSL(2,p^f)| <{p^{3f}}\le {({q_0}^2 +1)^{6}}$.
    Moreover, since $| {\OO^+(12,{q_0})}| = {q_0}^{30} ({q_0}^6-1)({q_0}^2 - 1)({q_0}^4 - 1)\cdots({q_0}^{10}-1)/{{\gcd(4,{q_0}^6 - 1)}}>({q_0}^2 +1)^{6}$,
    it follows that $|G/N|>|G|$, also a contradiction.

    When $m = 2$ or $m = 1$, we have $p^f\leq({q_0}+1)^n$, and so $| G |=|\PSL(2,p^f)| < ({q_0}+1)^{3n}$.
    Since $|G/N|=|\OO^+(2n,q_0)|> {q_0}^{2n(n-1)}$,  it can be observed that $|G/N|>|G|$,  a contradiction.

    Subcase 4.5. $G/N \cong \OO^-(2n,q_0)$, where $n \geq 2$.

    In terms of  Table 1,  $e_L(x)=[{2n}/{\rm{lcm}(2,x)}]-1$~ if  $x| n$;
    $e_L(x)=[{2n}/{\rm{lcm}(2,x)}]$ if $x\nmid n.$
    If $m| n$ and $m>2$,
    then $e_L(m)=[{2n}/{\rm{lcm}(2,m)}]-1$. Applying the inequality (3.7), it follows that $n(n-1)+1<2(m+1)([{2n}/{\rm{lcm}(2,m)}]-1)\leq 5n$, and so $n < 6$.
    Thus the pairs $(n,m)$ is one of $(3,3)$, $(4,4)$ and $(5,5)$.
    When $(n,m) =(3,3)$ or $(5,5)$, it's clear that ${p^f} = {\Phi _m}{({q_0})^{[{2n}/{\rm{lcm}(2,m)}]-1}} = 1$,
    which contradicts  $p>3$.
  Given $(n,m) = (4,4)$, then
  $|G/N|=|\OO^-(8,q_0)|={q_0}^{12}({q_0}^8 - 1)({q_0}^2 - 1)({q_0}^6-1)/{{\gcd(4,{q_0}^4 + 1)}}$.
    Moreover, $p^f=({q_0}^2 + 1)_p$ by (3.6), and so $| G | =|\PSL(2,p^f)|<p^{3f}\le {({q_0}^2 + 1)^3}$.
    It's obvious that $|G/N|>({q_0}^2 + 1)^3$, a contradiction.

    If $m>2$ and $m\nmid n$, then $e_L(m)=[{2n}/{\rm{lcm}(2,m)}]$.
    It follows that $n(n-1)+1<2(m+1)[{2n}/{\rm{lcm}(2,m)}] \leq 5n$, and so $n < 6$,
    the pairs $(n,m)$ is one of $(3,4)$, $(4,3)$, $(4,6)$, $(3,6)$, $(4,8)$, $(5,4)$, $(5,6)$, $(5,8)$ or $(5,10)$.
    It's not difficult to prove that  $|\OO^-(2n,q_0)|>|G|$.
    First of all, we discuss the case $e_L(m)=1$, i.e. $(n,m) =(3,4)$, $(4,3)$, $(4,6)$, $(3,6)$, $(4,8)$, $(5,6)$, $(5,8)$ or $(5,10)$.
    Then
    $|G/N|=| {\OO^-(2n,{q_0})}|\geq  {q_0}^6 ({q_0}^3+1)({q_0}^2 - 1)({q_0}^4 -1)/{{\gcd(4,{q_0}^3 + 1)}}$. By (3.6),
    it is clear that ${p^f} = {(\Phi_m(q_0))_p}\leq {q_0}^4+1$, and so
    $| G |=|\PSL(2,p^f)|<{p^{3f}}\le {({q_0}^4+1)^{3}}$.
    We can observe $|G/N|=|\OO^-(2n,q_0)|>|G|$, a contradiction.
     Secondly, if $e_L(m)=2$, then $(n,m) = (5,4)$. Since ${p^f} =
    ({q_0}^2+1)_p^2$, we have $| G | < {p^{3f}}\le {({q_0}^2+1)^{6}}$.
   On the other hand, $|G/N|=| {\OO^-(10,{q_0})}| = {q_0}^{20} ({q_0}^5+1)({q_0}^2 - 1)({q_0}^4 -
   1)({q_0}^6-1)({q_0}^8-1)/{{\gcd(4,{q_0}^5 + 1)}}$,
it follows $|\OO^-(10,q_0)|>({q_0}^2+1)^{6}$, and then $|G/N|>|G|$,
also a contradiction.

    When $m = 1$, we have $p^f=({q_0}-1)_p^n$, leading to $|G|=|\PSL(2,p^f)|< ({q_0}-1)^{3n}$.
    Moreover,
    $|G/N|=|\OO^-(2n,q_0)|={q_0}^{n(n-1)} ({q_0}^n+1)({q_0}^2 - 1)({q_0}^4 - 1)\cdots({q_0}^{2(n-1)}-1)/{{\gcd(4,{q_0}^n + 1)}}>({q_0}-1)^{3n}$.
    Therefore,  $|G/N|>|G|$, a contradiction.

    When $m = 2$, note that $p^f=({q_0}+1)_p^n$, so $|G|=|\PSL(2,p^f)|< {({q_0}+1)^{3n}}/{2} $ for ${q_0}\neq 2$.
    We observe that $|G/N|=|\OO^-(2n,q_0)|>|G|$, a contradiction.
     If $q_0=2$, then $|G/N|=2^{n(n-1)} (2^n+1) (2^n-1)\cdots (2^{2(n-1)}-1)$ and $|G|={3^n(3^{2n}-1)}/{2}$.
     Since $|G/N|\leq |G|$, it concludes that $n=2$. So $G/N\cong \OO^-(4,2)\cong \PSL(2,5)$ and $\mu(G)=\mu(\PSL(2,9))$. But $\mu_3(G)=\mu_3(\PSL(2,9))=2/9$ is less than
     $\mu_3(G/N)=\mu_3(\PSL(2,5))=1/3$, which contradicts the fact that $\mu_r(G)\geq \mu_r(G/N)$ in Lemma 2.1.

    Subcase 4.6. $G/N\cong {}^2{B_2(q_0)}$ with ${q_0} = {2^{2s + 1}}$.

    We know that $| {{}^2{B_2(q_0)}} | = {q_0}^2 ({q_0}^2 + 1) ({q_0} - 1)$.
    It follows from (3.6) that
    ${p^f} = {\Phi _4}{({q_0})}_p = {({q_0}^2 + 1)_p}$,
    hence ${q_0}^2 ({q_0} - 1)$ divides ${(({q_0}^2 + 1)_p^2 - 1)}/{2}$.
    If $p^f={q_0}^2+1$, then ${q_0}^2 ({q_0} - 1)$ divides ${(({q_0}^2 + 1)^2 - 1)}/{2}$,
    which implies that $2({q_0} - 1)| ({q_0}^2 + 2)$.
    However,  know that $\gcd({q_0} - 1,{q_0}^2 + 2) = 1$, a contradiction.
    If $p^f=({q_0}^2+1)_p<{q_0}^2+1$, then    $\gcd(({q_0}^2+1)_p+1,({q_0}^2+1)_p-1)\leq 2$.
    Moreover, since ${q_0} = {2^{2s + 1}}$,  5 divides ${q_0}^2+1$, and there exists a prime different from 5 in ${q_0}^2+1$, say s.
    When $p\neq 5$,  we have ${q_0}^2$ divides ${{(q_0}^2+1)}/{5}+1$ or ${{(q_0}^2+1)}/{5}-1$.
    However, it's clear that ${q_0}^2>{{(q_0}^2+1)}/{5}+1$, a contradiction.
    When $p=5$,
    it is evident that ${q_0}^2>{{(q_0}^2+1)}/{s}+1$, a contradiction.

    Subcase 4.7. $G/N\cong{}^3{D_4(q_0)}$.

    By the inequality (3.8), we have $13 < 2(m+1)e_L(m)$,  thus $m= 3, 6 , 12$.
    If $m= 3$,   then $p^f=({q_0}^2+{q_0}+1)_p^2$,
    which implies that $|G|=|\PSL(2,p^f)|<p^{3f}\le ({q_0}^2+{q_0}+1)^6$.
    But $|G/N|=|{}^3{D_4(q_0)}|={q_0}^{12}({q_0}^6-1)^2({q_0}^4-{q_0}^2+1)$.
    Obviously,
    $|G/N|>({q_0}^2+{q_0}+1)^6 >|G|$, a contradiction.
    If $m= 6$, then  $p^f=({q_0}^2-{q_0}+1)_p^2$, leading to $|G|<p^{3f}\le ({q_0}^2-{q_0}+1)^6$.
    It shows $|G/N|>|G|$, a contradiction.
    If $m=12$, we have $p^f=({q_0}^4-{q_0}^2+1)_p$, and then $|G|<p^{3f}\le ({q_0}^4-{q_0}^2+1)^3$.
    So $|G/N|>|G|$,
    a contradiction.

    Subcase 4.8. $G/N\cong {G_2(q_0)}$.

    Then there is the inequality $7< 2(m+1)e_L(m)$, then $m= 2, 3 , 6$.
    Let $m= 2$,  we obtain from (3.5) that $p^f={({q_0}+1)_p^2}$, which implies that $|G|=|\PSL(2,p^f)|< p^{3f}\le ({q_0}+1)^6$.
    Furthermore, note that $|G/N|=|{G_2(q_0)}|={q_0}^6({q_0}^6-1)({q_0}^2-1)$,
    it shows $|G/N|>|G|$, a contradiction.
    Let $m= 3$, we have $|G|<p^{3f}\le ({q_0}^2+{q_0}+1)^3$.
    It is easy to check  $|G/N|=|{G_2(q_0)}|>|G|$, a contradiction.
    If $m= 6$, then $|G|< ({q_0}^2-{q_0}+1)^3$.
   Clearly, $|G/N|=|{G_2(q_0)}|>|G|$,
    which contradicts to $|G/N|\big||G|$.

    Subcase 4.9. $G/N\cong {}^2{G_2(q_0)}$.

     Applying  (3.8), the inequality $4<2(m+1)e_L(m)$ holds for $m=2,6$.
    When $m= 2$, since $p^f=({q_0}+1)_p$,
    we have $|G|=|\PSL(2,p^f)|<p^{3f}\le ({q_0}+1)^3$.
    But $|{}^2{G_2(q_0)}|={q_0}^3({q_0}^3+1)({q_0}-1)$, it follows  $|G/N|>|G|$, a contradiction.
     If $m= 6$, then $|G|<p^{3f}\le ({q_0}^2-{q_0}+1)^3<|{}^2{G_2(q_0)}|=|G/N|$, a contradiction.

    Subcase 4.10. $G/N\cong  {F_4(q_0)}$.

    Similarly, we can obtain
    $25<2(m+1)e_L(m)$, so  $m=6,12$ by calculation.
    If $m=6$,  the equality (3.5) follows $p^f=({q_0}^2-{q_0}+1)_p^2$,
    leading to $|G|=|\PSL(2,p^f)|< ({q_0}^2-{q_0}+1)^6$.
    Since $|G/N|=|{F_4(q_0)}|={q_0}^{24}({q_0}^{12}-1)({q_0}^8-1)({q_0}^6-1)({q_0}^2-1)$,
    it shows that
    $|G/N|>|G|$, a contradiction.
    Let $m= 12$, we have $|G|<p^{3f}\le ({q_0}^4-{q_0}^2+1)^3$.
    Obviously, $|G/N|>|G|$,   still a contradiction.

    Subcase 4.11. $G/N\cong {{}^2F_4(q_0)}$.

    In light of (3.7), the inequality $13<2(m+1)e_L(m)$ holds for $m=4,6,12$.
    If $m= 4$,  we obtain from (3.5) that $p^f=({q_0}^2+1)_p^2$,
    and so $|G|=|\PSL(2,p^f)|<p^{3f}\le ({q_0}^2+1)^6$.
    Note that  $|G/N|=|{}^2{F_4(q_0)}|={q_0}^{12}({q_0}^6+1)({q_0}^4-1)({q_0}^3+1)({q_0}-1)$.
    It is obvious that  $ |G/N|>({q_0}^2+1)^6>|G|$,  a contradiction.
    When $m= 6$, we have $|G|<p^{3f}\le ({q_0}^2-{q_0}+1)^3$.
    It shows  $|G/N|> |G|$ ,
    a contradiction.
    If $m=12$, then $|G|<p^{3f}\le ({q_0}^4-{q_0}^2+1)^3<|G/N|$, a contradiction.

    Subcase 4.12. $G/N\cong E_6(q_0)$.

    From (3.6) we can get  $37<2(m+1)e_L(m)$, which there's no such  $m$ of Table 2 satisfying,  a contradiction.

    Subcase 4.13. $G/N\cong {}^2{E_6(q_0)}$.

    Now we have the inequality $37<2(m+1)e_L(m)$, then $m=6,18$.
    If $m=6$, then $p^f=({q_0}^2-{q_0}+1)_p^3$,
    which leads to  $|G|=|\PSL(2,p^f)|<p^{3f}\le ({q_0}^2-{q_0}+1)^9$.
    Furthermore, since $|G/N|=|{}^2{E_6(q_0)}|={q_0}^{36}({q_0}^{12}-1)({q_0}^9+1)({q_0}^8-1)({q_0}^6-1)({q_0}^5+1)({q_0}^2-1)/{\gcd(3,{q_0}+1)}>({q_0}^2-{q_0}+1)^9$,
    it implies that $|G/N|> |G|$,  a contradiction.
    If $m= 18$, then $|G|<p^{3f}\le ({q_0}^6-{q_0}^3+1)^3$.
    Clearly, $|G/N|>|G|$, also a contradiction.

    Subcase 4.14. $G/N\cong {E_7(q_0)}$.

    By Table 2, it is easy to check $64 >2(m+1)e_L(m)$ holds for any $m$,  which contradicts to (3.8).

    Subcase 4.15. $G/N\cong {E_8}(q_0)$.

    According to the inequality (3.8), we have $121 <2(m+1)e_L(m)$. But this inequality has  no solution in Table 2, a contradiction.

    Finally, we assume that $G/N$ is isomorphic to $^2F_4(2)'$ or a sporadic simple group whose order is known from  Table of ATLAS
    \cite{CCN}. Choose the largest order of Sylow subgroups of
    $G/N$, say $l$.
    Since $|G|=|\PSL(2,p^f)|={p^f(p^{2f}-1)}/{2}$ and $p^f\leq l$,
    we have $|G|\leq {l(l^2-1)}/{2}$. It is easy to check
    $|G/N|>{l(l^2-1)}/{2}$ when $G/N$ is isomorphic to $^2F_4(2)'$  or a sporadic simple group. So
    $|G/N|>|G|$, a contradiction.

    Therefore, $G\cong \PSL(2,q)$. This completes the proof.

\end{document}